\documentclass[11pt,reqno]{amsart}
\usepackage[margin=1.15in]{geometry}
\usepackage{amsmath,amssymb,amsthm}
\usepackage[T1]{fontenc}
\usepackage{lmodern,microtype}
\usepackage[colorlinks=true,linkcolor=blue,citecolor=blue,urlcolor=blue]{hyperref}
\newtheorem{theorem}{Theorem}[section]
\newtheorem{proposition}[theorem]{Proposition}
\newtheorem{corollary}[theorem]{Corollary}

\theoremstyle{definition}
\newtheorem{definition}[theorem]{Definition}

\theoremstyle{remark}
\newtheorem{remark}[theorem]{Remark}
\newcommand{\N}{\mathbb N}
\newcommand{\Z}{\mathbb Z}
\newcommand{\Q}{\mathbb Q}
\newcommand{\A}{\mathbb A}
\newcommand{\cO}{\mathcal O}
\newcommand{\Tor}{\operatorname{Tor}}
\newcommand{\Ret}{\mathcal R}
\title{An undecidable problem in dynamics}
\author{Jason Bell}
\address{Department of Pure Mathematics, University of Waterloo, Waterloo, Ontario, Canada}
\begin{document}
\begin{abstract} Given a variety $X$ with a self-map $\phi$ and subvarieties $Y$ and $Z$ of $X$, we consider the set of $n$ for which $\phi^n(Y)$ intersects $Z$ nontrivially.
Extending a beautiful construction of Lee and Nam, we show this gives rise to an undecidable problem for a fixed ambient variety $X$ and morphism $\phi$. We further use this construction to give a counterexample to a conjecture of the author, Satriano, and Sierra.
\end{abstract}
\maketitle

\section{Introduction}

Let $X$ be a variety over an algebraically closed field of characteristic zero,
let $\phi:X\to X$ be an endomorphism, let $Y\subseteq X$ be a closed
subvariety, and let $x\in X$.  The dynamical Mordell--Lang conjecture asserts
that the set
\[
\{n\in\N:\phi^n(x)\in Y\}
\]
is a finite, possibly empty, union of infinite arithmetic progressions along with a finite set.  One can interpret the assertion of the dynamical Mordell--Lang conjecture, roughly speaking, as saying that if the set of $n$ for which $\phi^n(x)$ lies in $Y$ is infinite then there should be a compelling geometric reason behind this phenomenon.  In particular, there should be a subvariety $Z$ of $Y$ that is forward invariant under some iterate of $\phi$ and some $a$ such that $\phi^a(x)\in Z$, so that $x$ is forced to continually return to the subvariety $Y$.

It is natural to ask what remains true when the point $x$ is replaced by a
positive-dimensional subvariety $W$.  There are at least two natural
analogues.  One can ask for the set of $n$ such that
\[
\phi^n(W)\subseteq Y,
\]
or one can impose the much weaker condition
\[
\phi^n(W)\cap Y\neq\varnothing.
\]
The first question retains much of the rigidity of ordinary dynamical
Mordell--Lang.  For example, for automorphisms of affine schemes of finite
type in characteristic zero, the author and Lagarias \cite{BL} proved that the set
of such $n$ is eventually periodic.

The second question, as it turns out, behaves very differently.  
Xie \cite[Question 69(i), p. 94]{Xie} (one can also see Question 9.13(i) of the arXiv version) asked whether, for a quasi-projective variety $X$ over $\mathbb C$ and irreducible closed subvarieties $W,Y\subseteq X$ satisfying
\[
\dim W+\dim Y<\dim X,
\]
the set
\begin{equation} \label{ret}
   \{n\in\N:\phi^n(W)\cap Y\neq\varnothing\}
\end{equation}
must again be a finite union of arithmetic progressions, with
singleton sets allowed. We denote the set in \eqref{ret} by
$\Ret_\phi(W,Y)$ and call it the \emph{return set} for $W$ with
respect to $\phi$ and $Y$, or more generally a \emph{dynamical
return set}. Images are taken set-theoretically, and intersections
are taken over the algebraically closed base field. For data
defined over $\Q$, we take this field to be $\overline{\Q}$.

Yang and Zhong
recently studied special cases of this question \cite{YZ}.  They also record
a striking example due to Jungin Lee and GyeongHyeon Nam, communicated privately to Xie, for
which the return set is precisely the set of composite numbers
\cite[Example 1.12]{YZ}. Thus the eventual periodicity predicted
by dynamical Mordell--Lang can fail even for this modest
generalization.

The purpose of this note is to show that the phenomenon in the
Lee--Nam example goes considerably further. A slight modification
of their construction can encode arbitrary polynomial equations
with bounded nonnegative integer unknowns. This gives much more than
a source of nonperiodic return sets, and now yields an undecidable
nonemptiness problem.

For $P\in\Z[X_1,\ldots,X_m,T]$, put
\begin{equation}\label{eq:bounded}
B_P=\{n\in\N:\text{there is } (a_1,\ldots , a_m)\in\{0,\ldots,n\}^m
\text{ with }P(a_1,\ldots ,a_m,n)=0\}.
\end{equation}
The bound on the unknowns is important and we note that for a fixed $P$, the set $B_P$
is computable, since membership of a given $n$ in $B_P$ can be checked by a finite search.
The true difficulty lies instead in deciding uniformly whether such a set is
nonempty.

Our first result realizes every such bounded Diophantine set as a return set.

\begin{theorem}\label{thm:intro}
Let $m\geq1$ and $P\in\Z[X_1,\ldots,X_m,T]$. There exist a
quadratic birational endomorphism $f$ of $\A^{2m+1}_{\Q}$,
a closed subvariety $Z\simeq\A^m$, and a closed subset $V_P$
defined explicitly from $P$, such that
\[
\{n\colon f^n(Z)\cap V_P\neq \varnothing\}=B_P.
\]
The map $f$ and subset $Z$ depend only on $m$, not on $P$.
\end{theorem}

We point out that the maps and subvarieties are completely explicit. Combining this
construction with a universal Diophantine equation gives
undecidability even when the map $f$ and starting subvariety $Z$ are fixed.

\begin{theorem}\label{thm:undec-intro}
There is a fixed quadratic birational endomorphism $f$ of
$\A^{17}_{\Q}$ and a fixed closed subvariety $Z$ for which there is
no algorithm deciding, from the equations of a closed subset $V$ over
$\Q$, whether the return set $\{n\colon f^n(Z)\cap V\neq \varnothing\}$ is empty.  \end{theorem}

Undecidability is of course familiar in dynamical and matrix problems, but it
often enters through the presence of several maps.  For
example, Paterson's matrix mortality problem asks whether some word in a
finite set of integer matrices has zero product, and is undecidable
\cite{Paterson}.  Similarly, Scanlon and Yasufuku \cite{SY} realize
exponential-polynomial sets as return sets using several
commuting endomorphisms.  With a single map, however, there is much
less room in which to encode computation. Indeed, for a single
linear map with rational input data, the ``point-to-point'' orbit
problem is decidable in polynomial time \cite{KL}, while deciding
whether the orbit of a specified rational point ever meets a
rational linear hyperplane is equivalent to the classical Skolem
problem and remains open. Recently, Karimov, Kelmendi, Ouaknine,
and Worrell \cite{KKOW} obtained an undecidability result for a
single linear map by asking for \emph{multiple} reachability:
whether there is a point in a prescribed semialgebraic set
whose orbit meets a target hyperplane in at least a fixed number of
distinct points. Theorem~\ref{thm:undec-intro} gives undecidability
with a single map and a single 
condition on nonempty intersections.  The additional freedom needed to encode difficult Diophantine problems
comes from allowing us to work with the orbit of a
positive-dimensional subvariety instead of a point.

For affine data
given by polynomial equations over $\Q$, membership in the return
set is decidable. Indeed, for a specified $n$, one can compute $f^n$ and
form the ideal
\[
I_Z+(f^n)^*I_V.
\]
The intersection $f^n(Z)\cap V$ is empty over $\overline{\Q}$ if
and only if this ideal contains $1$, which can be decided by a
Gr\"obner-basis computation. 
Thus both the return set and its complement are computable, and the same
argument works more generally for quasi-projective varieties by using a
finite affine cover.
What cannot be decided uniformly is
whether a return occurs at \emph{some} time. In fact, a related
construction gives families in which every return set is either
empty or a tail of $\N$, and there is still no algorithm deciding
which possibility occurs.

There is also a connection with a conjecture of the author, Satriano, and
Sierra concerning dynamical Mordell--Lang for coherent sheaves
\cite[Conjecture 1.2]{BSS}.  Our conjecture was that the eventual periodicity asserted in the dynamical Mordell--Lang conjecture should extend more generally to homological quantities that measure failure of varieties to intersect ``as expected''.
The same basic construction of Lee and Nam, together with a
direct homological calculation, gives the following counterexample to the
endomorphism case of our conjecture.

\begin{theorem}\label{thm:tor-intro}
There exist a quadratic birational endomorphism
$f:\A^3_{\Q}\to\A^3_{\Q}$ and smooth closed curves $Z,V$, each isomorphic
to $\A^1$, such that
\[
\left\{n\in\N:
\Tor^{\cO_{\A^3}}_1((f^n)^*\cO_V,\cO_Z)\ne0\right\}
=\{r^2:r\geq1\}.
\]
In particular, the endomorphism case of that conjecture is false.
\end{theorem}

The paper is organized as follows. In Section~\ref{sec:tlc} we
look at a class of sets we call TLC sets, which we'll later show arise as return sets. In Section~\ref{sec:waterfalls} we give the construction of Lee and Nam
and prove Theorem~\ref{thm:intro}. Section~\ref{sec:jones} uses a
universal Diophantine equation to fix the ambient dimension and
proves Theorem~\ref{thm:undec-intro}. In Section~\ref{sec:tails} we
show that undecidability persists even when one restricts to a setting in which every return set is
empty or a tail of $\N$. Finally, in Section~\ref{sec:tor} we give
the coherent-sheaf counterexample.

\section{Largest coordinates of Diophantine solutions}\label{sec:tlc}

Before turning to dynamics, we isolate the Diophantine mechanism underlying
the construction of Lee and Nam.  The key point is that we can create a dynamical system $(X,\phi)$ with subvarieties $W$ and $Y$ whose return times are precisely the largest coordinates of nonnegative integer zeros of a multivariate integer polynomial.

The Davis--Putnam--Robinson--Matiyasevich theorem (see, for example,
\cite{Mat}) shows that unrestricted
existential Diophantine conditions can encode arbitrary computably enumerable
sets.  If instead one takes the largest coordinate of a Diophantine solution, the behaviour is quite
different, and membership in the resulting set can always be decided by a finite search.  For
bookkeeping purposes, we give these sets a name.  The word ``Turing'' is
meant only to emphasize this computability-theoretic aspect.

We use the convention $\N=\{0,1,2,\ldots\}$.

\begin{definition}
A subset $S$ of $\N$ is a \emph{Turing largest-coordinate set}, or
\emph{TLC set}, if there exist $r\geq1$ and a polynomial
\[
Q\in\Z[X_1,\ldots,X_r]
\]
such that
\[
S=\{\max(a_1,\ldots,a_r):
(a_1,\ldots,a_r)\in\N^r,\ Q(a_1,\ldots,a_r)=0\}.
\]
\end{definition}

We note that every TLC set $S$ is computable since to decide whether $n\in S$, it suffices
to test the finitely many tuples in $\{0,\ldots,n\}^r$ having largest
coordinate $n$.  The restriction to largest coordinates is important here, since an unrestricted existential projection of the zero set of an integer
polynomial is able to encode all computably enumerable subsets of the natural numbers.

The following observation gives the form of TLC sets that will be useful for
the dynamical construction.

\begin{proposition}\label{prop:tlc}
A subset $S$ of $\N$ is a TLC set if and only if there exist $m\geq1$ and
a polynomial
\[
P\in\Z[X_1,\ldots,X_m,T]
\]
such that
\[
S=B_P
=\{n\in\N:\text{there is }
(a_1,\ldots,a_m)\in\{0,\ldots,n\}^m
\text{ with }P(a_1,\ldots ,a_m,n)=0\}.
\]
\end{proposition}

\begin{proof}
Suppose first that $S$ is a TLC set, so there is a polynomial $Q\in\mathbb{Z}[X_1,\ldots ,X_r]$ such that
\[
S=\{\max(a_1,\ldots,a_r):
(a_1,\ldots ,a_r)\in\N^r,\ Q(a_1,\ldots ,a_r)=0\}.
\]
If $r=1$, take $m=1$ and $P(X_1,T)=Q(T)$. Suppose now that
$r\geq2$, and define
\[
P(X_1,\ldots,X_{r-1},T)
:=
\prod_{\sigma\in S_r}
Q(X_{\sigma(1)},\ldots,X_{\sigma(r)})\Big|_{X_r=T}.
\]

Here $S_r$ denotes the symmetric group. If $Q$ vanishes at a
nonnegative integer tuple with largest coordinate $n$, permuting
a coordinate equal to $n$ into the last position gives a zero
of $P$ with all remaining coordinates at most $n$. Conversely,
a vanishing factor of $P(a_1,\ldots,a_{r-1},n)$ gives a zero of
$Q$ whose largest coordinate is $n$. Thus $S=B_P$.

Conversely, suppose that $S=B_P$ for some
$P\in\Z[X_1,\ldots,X_m,T]$.  Introduce new variables
$Y_1,\ldots,Y_m$ and set
\[
Q(X_1,\ldots ,X_m, Y_1,\ldots, Y_m,T)
=
P(X_1,\ldots ,X_m,T)^2
+
\sum_{i=1}^m(X_i+Y_i-T)^2.
\]
A nonnegative integer zero of $Q$ satisfies
\[
P(X_1,\ldots ,X_m,T)=0,
\qquad
X_i+Y_i=T\quad(1\leq i\leq m).
\]
Hence every $X_i$ and $Y_i$ is at most $T$, so $T$ is the largest coordinate.
Conversely, if
\[
P(a_1,\ldots ,a_m,n)=0,
\qquad
0\leq a_i\leq n,
\]
then taking $Y_i=n-a_i$ gives a nonnegative integer zero of $Q$ whose largest
coordinate is $n$.  Thus the set of largest coordinates of the nonnegative
integer zeros of $Q$ is exactly $B_P$.
\end{proof}

The next section shows that the
Lee--Nam construction realizes every TLC set as a dynamical return set.

\section{Don't go chasing waterfalls}\label{sec:waterfalls}
The original motivation behind this paper was to ask whether there
might still be some reasonably well-behaved description of the
return sets in Xie's general dynamical Mordell--Lang question.
The Lee--Nam construction shows that these sets can encode complicated Diophantine sets and so this quest was more complex than I had originally envisioned.

Let $k$ be an algebraically closed field of characteristic zero. We use affine coordinates
\[
(t,x_1,\ldots,x_m,y_1,\ldots,y_m)
\]
on $\A_k^{2m+1}$, and we define the endomorphism
\begin{equation}\label{eq:map}
f(t, x_1,\ldots ,x_m,y_1,\ldots ,y_m)
=\bigl(t+1, x_1,\ldots ,x_m,
y_1(x_1-t-1),\ldots,y_m(x_m-t-1)\bigr).
\end{equation}
We take
\[
Z=V(t,y_1-x_1,\ldots,y_m-x_m),\qquad
V_P=V(y_1,\ldots,y_m,P(x_1,\ldots ,x_m ,t)).
\]
We note that these quantities are defined over $\Q$ when $P$ has rational coefficients and more generally over any subfield of $k$ containing all coefficients of $P$.

\begin{proposition}[The Lee--Nam machine]\label{prop:realization}
For $P\in\Z[ X_1,\ldots ,X_m,T]$, one has $$\{n\colon f^n(Z)\cap V_P \neq \varnothing\}=B_P.$$
In particular, every TLC set can be realized as a return set for a
quadratic birational endomorphism of affine space.
\end{proposition}
\begin{proof}
Write
\[
H_n(X)=\prod_{j=0}^{n}(X-j).
\]
Induction gives
\[
f^n(0, x_1,\ldots, x_m,x_1,\ldots, x_m)
=\bigl(n, x_1,\ldots, x_m,H_n(x_1),\ldots,H_n(x_m)\bigr).
\]
Thus $f^n(Z)\cap V_P$ is nonempty precisely when there are
$x_1,\ldots,x_m\in k$ such that
\[
H_n(x_1)=\cdots=H_n(x_m)=0,
\qquad P(x_1,\ldots,x_m,n)=0.
\]
Since the roots of $H_n$ in $k$ are exactly $0,1,\ldots,n$, this is equivalent to $n\in B_P$.
\end{proof}
\begin{remark}
The reader might wonder whether the pathological return set is somehow an
artefact of $f^n(Z)$ being merely constructible, and whether the situation
would change upon replacing it by its closure.  In fact, the computation
above gives
\[
f^n(Z)
=
V\bigl(t-n,y_1-H_n(x_1),\ldots,y_m-H_n(x_m)\bigr),
\]
so $f^n(Z)$ is already closed.
\end{remark}

We note that the map $f$ above is in fact birational with rational inverse given by
\[
(t,x_1,\ldots ,x_m,y_1,\ldots ,y_m)\longmapsto
\left(t-1,x_1,\ldots ,x_m,
\frac{y_1}{x_1-t},\ldots,\frac{y_m}{x_m-t}\right).
\]
It is not an automorphism, however, since the inverse is merely rational and not regular. 

For nonconstant $P\ne0$, the dimensions of $Z$ and $V_P$ constructed above satisfy
$\dim Z=m$ and $\dim V_P\leq m$, so their sum is less than $2m+1$.

\begin{remark}\label{rem:smooth}
The closed set $V_P$ need not be irreducible in general, which is a constraint in Xie's formulation \cite{Xie}, but one extra coordinate
allows us to arrange that both subvarieties are smooth and
irreducible. We extend $f$ to
\[
f':=f\times\mathrm{id}_{\A^1}
\]
on $\A^{2m+2}$, with last coordinate $z$, and put
\[
\begin{aligned}
Z'
&=V(t,y_1-x_1,\ldots,y_m-x_m,z),\\
V'_P
&=V(y_1,\ldots,y_m,z-P(x_1,\ldots ,x_m,t)).
\end{aligned}
\]
Then $Z'\simeq\A^m$ and
$V'_P\simeq\A^{m+1}$, with the latter being the graph of $P$
over the $( x_1,\ldots ,x_m,t)$ coordinates. Since $z=0$ on every
iterate of $Z'$, the intersection conditions are exactly
those in Proposition~\ref{prop:realization}. Hence the return set for $Z'$ with
respect to $f'$ and $V'_P$ is precisely $B_P$, and
\[
\dim Z'+\dim V'_P=2m+1<2m+2.
\]
Moreover, $ f'$ and $Z'$ still depend only on $m$.
\end{remark}

\begin{corollary}\label{cor:h10}
There is no algorithm deciding nonemptiness of return sets for endomorphisms and closed subsets given by equations over $\Q$. This remains true for quadratic birational endomorphisms.
\end{corollary}
\begin{proof}
Given $Q\in\Z[X_1,\ldots,X_m]$, apply Proposition~\ref{prop:realization} to $P(X_1,\ldots ,X_m,T)=Q(X_1,\ldots ,X_m)$. Then using $f, Z$, and $V_P$ as in the Proposition
\[
\{n\colon f^n(Z)\cap V_P\neq \varnothing\}\neq \varnothing
\quad\Longleftrightarrow\quad
Q\text{ has a zero in }\N^m.
\]
The latter problem is undecidable by the Davis--Putnam--Robinson--Matiyasevich theorem; see \cite{Mat}.
\end{proof}

\section{A fixed dynamical system}\label{sec:jones}

It is useful first to recall an intuitive way of thinking about the halting problem in computer science.
If you work with your favourite programming language, then given a program along with some initial
input, we can ask whether the resulting computation of running the program on this input eventually terminates.
There is no algorithm which answers this question correctly for every
program and every input.

On the other hand, there is no difficulty in writing a single program which
takes another program and its input as data and simply runs that program on that input.
This universal program halts on the given data precisely when the
program it is simulating halts. Thus all of the complexity of the
halting problem can be placed inside one fixed machine, with only
its input varying.

The Davis--Putnam--Robinson--Matiyasevich theorem gives a remarkable
Diophantine analogue of this observation.  After encoding programs and their
inputs by nonnegative integers, there is a single integer polynomial
\[
U(A,X_1,\ldots,X_r)
\]
such that
\[
a\in K
\quad\Longleftrightarrow\quad
U(a,u_1,\ldots,u_r)=0
\]
for some $u_1,\ldots,u_r\in\N$, where $K$ is a fixed version of the
halting set. Such a polynomial can be constructed effectively, where
the parameter $a$ encodes the program and its input, while the
remaining variables provide a Diophantine certificate that the
corresponding computation halts.

A theorem of Matiyasevich, for which Jones gave a complete proof in
\cite{Jones}, shows that one may take $r=9$.  Thus the undecidability of the
halting problem can be packaged into the existence of nonnegative integer
zeros of one fixed polynomial in ten variables, one of which is regarded as
a parameter.

We now combine this with the Lee--Nam construction to prove Theorem~\ref{thm:undec-intro}.
\begin{proof}[Proof of Theorem~\ref{thm:undec-intro}]
We let $K$ be a fixed version of the halting set and we now fix a polynomial $U$ as above, so that
\begin{equation}\label{eq:jones}
a\in K
\quad\Longleftrightarrow\quad
\exists (u_1,\ldots,u_9)\in\N^9
\quad U(a,u_1,\ldots,u_9)=0.
\end{equation}

For each $a$, form
\[
P_a(X_1,\ldots,X_8,T)=
\prod_{\sigma \in S_9}U(a,X_{\sigma(1)},\ldots,X_{\sigma(9)})\Big|_{X_9=T}.
\]
By the first half of Proposition~\ref{prop:tlc}, $B_{P_a}$ is precisely the set of largest coordinates of zeros of $U(a,X_1,\ldots, X_9)$ in $\N^9$. Hence
\[
a\in K\quad\Longleftrightarrow\quad B_{P_a}\ne\varnothing.
\]
Taking $m=8$ in Proposition~\ref{prop:realization} gives the desired result.
\end{proof}
\begin{remark}
Only the closed subset $V$ varies in Theorem \ref{thm:undec-intro}.  The map $f$ and the closed subset $Z$ are fixed.  Once the polynomial $U$ is fixed, even the degrees of the polynomial equations defining $V$ are bounded independently of $a$. 
\end{remark}

Applying Remark~\ref{rem:smooth} with $m=8$ gives the same
undecidability conclusion with ambient variety $\A^{18}$, and where now both $V$ and $Z$ can be taken to be smooth and irreducible, and with the sum of their
dimensions strictly smaller than the ambient dimension.

\section{A sun that never sets}\label{sec:tails}

In fact, we show that the Lee--Nam construction can be used to give an even stronger phenomenon:
once a return occurs, every later time is also a return, and yet there is no
algorithm for deciding whether a return ever occurs at all.

\begin{corollary}
There is a fixed quadratic birational endomorphism
\[
f:\A^{19}_{\Q}\to\A^{19}_{\Q}
\]
and a fixed smooth closed subvariety $Z$ such that, as $V$ varies through an
explicitly defined family of closed subsets, every return set
\[
\{n\colon f^n(Z)\cap V\neq\varnothing\}
\]
is either empty or a tail of $\N$, and there is no algorithm for deciding
which of these two possibilities occurs.
\end{corollary}
\begin{proof}
Let
\[
U(A,X_1,\ldots,X_9)\in\Z[A,X_1,\ldots,X_9]
\]
be the fixed polynomial from the preceding section, so that for some fixed undecidable computably enumerable set $K\subseteq\N$ (for example, one can take a standard encoding of the halting problem) we have
\[
a\in K
\quad\Longleftrightarrow\quad
U(a,u_1,\ldots ,u_9)=0
\]
for some $(u_1,\ldots ,u_9)\in\N^9$.

We use the Lee--Nam construction with nine bounded variables.  Thus, in
coordinates
\[
(t,x_1,\ldots,x_9,y_1,\ldots,y_9)
\]
on $\A^{19}$, let
\[
f(t, x_1,\ldots ,x_9, y_1,\ldots, y_9)
=
\bigl(t+1,x_1,\ldots ,x_9,
y_1(x_1-t-1),\ldots,y_9(x_9-t-1)\bigr),
\]
and let
\[
Z=V(t,y_1-x_1,\ldots,y_9-x_9).
\]
Both $f$ and $Z$ are fixed.

For each $a\in\N$, regard
\[
U(a,X_1,\ldots,X_9)
\]
as a polynomial in $X_1,\ldots,X_9,T$ which is independent of $T$, and put
\[
V_a
=
V\bigl(y_1,\ldots,y_9,U(a,x_1,\ldots ,x_9)\bigr).
\]
By the Lee--Nam construction,
\[
\Ret_f(Z,V_a)
=
\{n\in\N:
\text{there is }(u_1,\ldots ,u_9)\in\{0,\ldots,n\}^9
\text{ with }U(a,u_1,\ldots ,u_9)=0\}.
\]

If $a\notin K$, this set is empty.  Suppose now that $a\in K$, and let
\[
b_a
=
\min\{\max_i u_i:
U(a,u_1,\ldots ,u_9)=0,\ (u_1,\ldots ,u_9)\in\N^9\}.
\]
Then a solution $(u_1,\ldots ,u_9)$ realizing this minimum belongs to
$\{0,\ldots,n\}^9$ exactly when $n\geq b_a$.  Hence
\[
\Ret_f(Z,V_a)=\{n\in\N:n\geq b_a\}.
\]

There is also a direct dynamical reason for this persistence.  Since the
$x_i$ are fixed by $f$, the polynomial $U(a, x_1,\ldots ,x_9)$ remains zero
along the orbit of a point of $V_a$, while $y_i=0$ implies that the new
$y_i$ coordinate is again zero.  Thus
\[
f(V_a)\subseteq V_a.
\]
Once an iterate of $Z$ meets $V_a$, every subsequent iterate does as well.

Consequently every return set in this family is either empty or a tail of
$\N$.  If there were an algorithm deciding which alternative occurs, then
we could decide whether $a\in K$, a contradiction.
\end{proof}

Thus we see nonemptiness, finiteness, infinitude, and cofiniteness are all undecidable uniformly, even among these examples whose return sets are individually eventually periodic. 

\section{A counterexample for coherent sheaves}\label{sec:tor}

In earlier joint work with Matt Satriano and Sue Sierra, we considered
a version of dynamical Mordell--Lang for coherent sheaves in which ordinary
intersection is replaced by the vanishing of higher Tor groups.  Recall that, for subvarieties $Y,Z$ meeting properly in a smooth ambient variety $X$, Serre’s intersection formula expresses their local intersection multiplicities as alternating sums of lengths of Tor modules $\Tor_i^X(\cO_Y,\cO_Z)$,
where $\cO_Y$ and $\cO_Z$ are the structure sheaves associated to $Y$ and $Z$ respectively \cite[Appendix A, p.~426]{Hartshorne}.

The vanishing of the higher Tor groups then says that $Y$ and $Z$ are
\emph{homologically transverse}, which can be viewed as an algebraic proxy for the
statement that their intersection behaves as expected.

For example, two curves in $\A^3$ are generically disjoint.  If
\[
Y=V(y,z),\qquad Z=V(x,y),
\]
then they meet at the origin, which is in some sense unexpected.  Writing
$R=k[x,y,z]$, one has
\[
\Tor_1^R(R/(y,z),R/(x,y))
 \simeq
 \frac{(y,z)\cap(x,y)}{(y,z)(x,y)},
\]
and the image of $y$ in the right-hand side is nonzero in the quotient.  In this particular case, the nonvanishing of the first Tor group
is detecting this intersection.

It turns out that, although this homological framework contains finer information than
set-theoretic intersection, it still has substantial overlap with the
variant of dynamical Mordell--Lang considered by Xie.  The Lee--Nam
construction gives a particularly quick counterexample to the main
conjecture of \cite{BSS} in the endomorphism case.

\begin{proof}[Proof of Theorem~\ref{thm:tor-intro}]  We work over an algebraically closed
field $k$ of characteristic zero and put
$R=k[t,x,y],$
which we regard as the coordinate ring of $\A^3_k$.  Consider
\[
f(t,x,y)=(t+1,x,y(x-t-1)),
\]
and let $Z$ and $V$ be the curves defined respectively by the ideals
\[
J=(t,y-x)
\qquad\text{and}\qquad
(y,t-(x+1)^2).
\]
Both are smooth curves isomorphic to $\A^1$.

Proposition~\ref{prop:realization}, applied to
\[
P(X,T)=T-(X+1)^2,
\]
already shows that
\[
f^n(Z)\cap V\ne\varnothing
\]
precisely when $n$ is a positive perfect square.  We show that exactly the
same set of return times is detected by the first Tor group.

For $n\geq0$, put
\[
F_n(t,x)=\prod_{j=1}^n(x-t-j),
\qquad
g_n(t,x)=t+n-(x+1)^2,
\]
where $F_0=1$.

Since $\mathbb{A}^3={\rm Spec}(R)$, we can identify coherent sheaves with their corresponding
finitely generated $R$-modules.  A direct induction shows that $(f^n)^*(\cO_V)$ is identified with $R/I_n$
with
$I_n=(yF_n,g_n)$
and
$\cO_Z$ is identified with $R/J$.

Thus
\[
\Tor_1^{\cO_{\A^3}}\bigl((f^n)^*(\cO_V),\cO_Z\bigr)
\]
is nonzero if and only if
\[
\Tor_1^R(R/I_n,R/J)\cong (I_n\cap J)/I_nJ
\]
is nonzero.  Since $f^n(Z)\cap V$ is empty when $n$ is not a positive perfect square, we see that in this case $I_n+J=R$ and so $I_n\cap J=I_nJ$ and $\Tor_1^R(R/I_n,R/J)$ vanishes in this case.  

Next suppose that $n=(j+1)^2$ with $j\geq 0$. Modulo $J$ we have $t=0$ and $y=x$, so the generators \(yF_n\) and \(g_n\) of \(I_n\) become
\[
h_n(x)=x(x-1)\cdots(x-n)
\]
and
\[
q_n(x)=n-(x+1)^2=-(x-j)(x+j+2)
\]
when we work mod $J$.
Write
\[
h_n=(x-j)p_n,\qquad q_n=(x-j)s_n,
\]
where \(s_n=-(x+j+2)\). Since the roots of \(p_n\) are contained in
\(\{0,\ldots,n\}\setminus\{j\}\), the polynomials \(p_n\) and \(s_n\) are coprime.

Set
\[
u=s_nyF_n-p_ng_n\in I_n.
\]
Modulo $J$ this becomes
\[
s_nh_n-p_nq_n=0,
\]
so $u\in I_n\cap J$. We claim that $u\notin I_nJ$. To see this, note that if this were not the case, we could find $\alpha,\beta\in J$ such that
\[
u=\alpha yF_n+\beta g_n,
\]
and hence
\[
(s_n-\alpha)yF_n=(p_n+\beta)g_n.
\]
Since $yF_n$ and $g_n$ are coprime in $R$, it follows that
$g_n$ divides $s_n-\alpha$, say
\[
s_n-\alpha=cg_n.
\]
Reducing modulo $J$ gives $q_n=(x-j)s_n$ divides $s_n$ in $k[x]$.
Since $s_n\ne0$, this is impossible. Thus $u\notin I_nJ$, and consequently
\[
\Tor_1^R(R/I_n,R/J)\ne0.
\]

We have now proved that
\[
\{n\in\N:
\Tor_1^{\cO_{\A^3}}\bigl((f^n)^*\cO_V,\cO_Z\bigr)\ne0\}
=
\{r^2:r\geq1\},
\]
which gives us our counterexample to \cite[Conjecture 1.2]{BSS} for
endomorphisms.  We note that the data are defined over $\mathbb Q$, and both vanishing and nonvanishing of the relevant Tor modules is preserved by extension of scalars to $\overline{\mathbb Q}$, so we have proved Theorem \ref{thm:tor-intro}.
\end{proof}

This example does not address the automorphism case of \cite[Conjecture 1.2]{BSS}, since while the map $f$ above is a birational morphism, its inverse is only rational.  The
positive results of \cite[Theorem 1.3]{BSS}, for smooth surfaces and for
automorphisms lying in an algebraic group action, are therefore unaffected.
It would be interesting to know whether the automorphism version of our conjecture holds.

%\begin{remark}
%The example above uses the same mechanism as the Lee--Nam example
%in \cite[Example 1.12]{YZ}. Their construction realizes the composite
%numbers as a return set in $\A^5$. Here we obtain the positive squares
%using two smooth curves in $\A^3$, and the calculation above shows
%that the same return times are detected by the first Tor group.
%\end{remark}
\section*{AI declaration}

What a bizarre world we live in! The main idea behind this paper comes from a beautiful example of two human mathematicians, Jungin Lee and GyeongHyeon Nam. At some point, I realized that their example could be modified to give an undecidable problem in dynamics involving a single map and, in addition, a counterexample to a conjecture in a paper I had written years ago with Matt Satriano and Sue Sierra.

ChatGPT offered me one free month of what they call a ``Plus'' subscription. I took them up on it and, as an experiment, gave it some of my rough mathematical thoughts on this problem and asked it to turn them into a paper. Aside from the fact that I could not initially understand much of what it had written, it was largely a success. I then rewrote the paper, while continuing to use ChatGPT to discuss exposition and read successive versions of the draft.

Interacting with the machine was genuinely fascinating, but it also left me with the eerie sense that, while we as a species have willingly chosen each individual step that has brought us here, we may not like where all of this ultimately leads. But anyway, I had the subscription, and I kept using it. We eventually converged on a version I liked, although it could not resist suggesting that I remove all of my footnotes.\footnote{It was probably right.}

For reasons I do not fully understand, while all this was going on I also asked it to explain some of the more subtle plot points of the movie \emph{Big Momma's House},\footnote{I saw this movie on an airplane either shortly before or shortly after 9/11. In the movie, there is a woman everyone refers to as ``Big Momma''. I was curious how such a situation could arise unless the woman insisted that she be called ``Big Momma'' during all introductions, but ChatGPT explained how such a phenomenon can arise organically.} a comedy from the year 2000 in which Martin Lawrence plays an FBI agent who disguises himself as a heavyset grandmother as part of an undercover operation to recover stolen bank funds. Strangely, the universe of \emph{Big Momma's House} now makes more sense to me than our own.

\end{document}